\documentclass[11pt]{article}

\usepackage{amssymb}
\usepackage{amscd}
\usepackage{amsmath}
\usepackage{amsfonts}
\usepackage{amsthm}
\usepackage{color}
\usepackage{enumerate}

\usepackage[pdftex]{hyperref}

\theoremstyle{plain}

\newtheorem{theo}{Theorem}[section] 
\newtheorem{prop}[theo]{Proposition}
\newtheorem{lemma}[theo]{Lemma}
\newtheorem{cor}[theo]{Corollary}

\theoremstyle{definition}

\newcommand{\R}{{\mathbb{R}}}
\newcommand{\N}{{\mathbb{N}}}
\newcommand{\C}{{\mathbb{C}}}
\newcommand{\Z}{{\mathbb{Z}}}

\newcommand{\be}{{\beta}}
\newcommand{\al}{{\alpha}}
\newcommand{\la}{{\lambda}}

\newcommand{\si}{{\sigma}}

\newcommand{\ga}{{\gamma}}
\newcommand{\te}{{\theta}}
\newcommand{\om}{{\omega}}
\newcommand{\Om}{\Omega}

\newcommand{\Ci}{{\mathcal{C}}^{\infty}} 
\newcommand{\Cl}{\mathcal{C}}
\newcommand{\op}{\operatorname}

\newcommand{\bigo}{\mathcal{O}} 
\newcommand{\Hilb}{\mathcal{H}}

\newcommand{\Herm}{\op{Herm}}

\newcommand{\preq}{\op{Preq}}
\newcommand{\pathspace}{\mathcal{P}}
\newcommand{\ham}{\op{Ham}}

\newcommand{\Ham}{\op{Ham}}

\newcommand{\J}{\mathcal{J}}
\newcommand{\quant}{\mathcal{Q}}

\newcommand{\wt}{\widetilde}

\newcommand{\ca}{\operatorname{Cal}}
\newcommand{\Sh}{\operatorname{Sh}}

\begin{document}

\author{Laurent Charles}

\title{On a quasimorphism of Hamiltonian diffeomorphisms   and quantization}
\maketitle

\begin{abstract} In the setting of geometric quantization, we associate to any prequantum bundle automorphism a unitary map of the corresponding quantum space. These maps are controlled in the semiclassical limit by two invariants of symplectic topology: the Calabi morphism and a quasimorphism of the universal cover of the Hamiltonian diffeomorphism group introduced by Entov, Py, Shelukhin. 
\end{abstract}

\section{Introduction} 
Geometric quantization deals with defining a quantum system corresponding to a
given classical system, usually given with the Hamiltonian formalism
\cite{So}. From its introduction, it has been deeply connected to
representation theory and the first application \cite{Kos} was the orbit method: 
constructing irreducible representations of a Lie group by quantizing its
coadjoint orbits. Later on, geometric quantization  has been applied to flat
bundle moduli spaces to produce
projective representations of mapping class groups \cite{hitchin} \cite{AxWi}, which are of fundamental
importance in quantum topology.

Besides these achievements, semiclassical methods for geometric quantization
have been developed successfully after
the seminal work \cite{BoGu}, broadening the field of applications to any
Hamiltonian of a prequantizable compact symplectic manifold.

Our goal in this paper is to study a natural asymptotic representation for the group of
Hamiltonian diffeomorphisms, actually a central extension of this group, defined in the 
context of geometric quantization. Here the adjective ``asymptotic'' refers to
the fact that our representation satisfies the homomorphism equation up to an
error small in the
semiclassical limit. This limitation is inherent to the analytical
methods we use but
is also meaningful. Indeed,  the simplicity of the group of Hamiltonian
diffeomorphisms of a compact symplectic manifold \cite{Ba} imposes 
severe restrictions on the possible representations as was first noticed in
\cite{GiMo}. This simplicity explains that quasimorphisms are of big interest
in symplectic topology \cite{PoRo}.
Actually, our main result says that our asymptotic representation is controlled at
first order by a quasimorphism of the universal
cover of Hamiltonian diffeomorphism group introduced in \cite{En}, \cite{Py}, \cite{Sh}.

\subsection{The asymptotic representation}
let $M$ be a symplectic compact manifold equipped with a prequantum line
bundle $L$. The 
quantum space  will be defined as a subspace of $\Ci ( M , L)$ depending on
some auxiliary data. Typically, this additional data is a holomorphic
structure and the quantum space consists of the corresponding holomorphic
sections.
The group we will work with is
the group $\preq$ of prequantum bundle automorphisms of $L$. It acts naturally on $\Ci ( M , L)$ by
push-forward, but without preserving the quantum space. To remedy this, we will consider parallel  transport in the bundle of quantum
spaces along some specific paths.

Recall first that $L$ is a Hermitian line
bundle over $M$ equipped with a connection $\nabla$ whose curvature is
$\frac{1}{i}$ times the symplectic form $\om$. 
Any complex structure $j$ of $M$ compatible with $\om$ has a natural lift to a
holomorphic structure of $L$ determined by the condition that $\nabla$ becomes
the Chern connection. Denote by $\quant ( j) $ the corresponding space of
holomorphic sections of $L$.
The prequantum bundle
automorphisms of $L$ are the vector bundle automorphisms preserving the
Hermitian structure and the connection. The push-forward by a prequantum bundle
automorphism $\varphi$ sends $\quant (j)$ to $\quant ( \pi(\varphi)_* j)$,
where $\pi ( \varphi)$ is the diffeomorphism of $M$ lifted by $\varphi$. The
important observation is that  if we were able to identify equivariantly the various $\quant (j)$, we would have a representation of $\preq$.

One geometrical way to produce such an identification is to consider each
$\quant ( j)$ as a fiber of a bundle over the space of complex structures and
to introduce a flat equivariant connection. To do this, it is more convenient
to work with the space $\J $ of almost complex structures of $M$ compatible
with $\om$, because it is a smooth contractible (infinite dimensional)
manifold. Following \cite{GuUr}, we can still define $\quant ( j)$  for any $j
\in \J $ as a sum of some eigenspaces of a convenient Laplacian. Simplifying
slightly, this defines a vector subbundle $ \quant \rightarrow \J $ of $ \J
\times \Ci ( M , L) $. Then using the natural scalar product of $\Ci ( M ,
L)$, we obtain a connection of $\quant \rightarrow \J$, sometimes called the
$L^2$-connection. This connection is equivariant with respect to the action of
$\preq$, but unfortunately it is known not to be flat \cite{FoUr}.

Still we will use this connection to construct an application
\begin{gather} \label{eq:def_psi}
  \Psi: \preq \rightarrow \op{U} ( \Hilb  ),
\end{gather}
where $\Hilb = \quant (j_0)$, $j_0$ being a given base point of $\J $. We will need a particular family of paths of $\J$ which was introduced in \cite{Sh}. For any two points $j_0$, $j_1$ of $\J$, for any $x\in M$, $j_0(x)$ and $j_1(x)$ are linear complex structures of $(T_xM, \om_x)$.  The space of linear complex structures of a symplectic vector space has a natural Riemannian metric such that any two points are connected by a unique geodesic, and in particular there is a unique geodesic $j_t(x)$ joining $j_0(x)$ to $j_1(x)$. This defines a path $(j_t)$ of $\J$ that we call abusively the geodesic from $j_0$ to $j_1$. 

Let us now define the map $\Psi$. For any $\varphi \in \preq$, $\Psi ( \varphi)$ is the composition of the push-forward $\varphi_*: \quant ( j_0) \rightarrow \quant ( j_1)$, $j_1 =\pi(\varphi)_* j_0$, with the parallel transport along the geodesic joining $j_1$ to $j_0$.

The inspiration comes from the geometric construction of the quantum
representations of the mapping class group of a surface \cite{hitchin},
\cite{AxWi}. The connections used in these papers are projectively flat so
that the resulting representations are projective and the choice of paths does
not matter. To the contrary,   our result will  depend essentially on the choice of paths. The idea to use these particular geodesics in the context of geometric quantization is new.

Before we continue, let us introduce the {\em semiclassical limit}. For any positive
integer $k$, we replace in the previous definitions the bundle $L$ by
its $k$th tensor power $L^k$, which defines $\quant_k \rightarrow \J$, $\Hilb_k := \quant_k (j_0)$ and
 \begin{gather} \label{eq:def_psi_k}
  \Psi_k: \preq \rightarrow \op{U} ( \Hilb_k  ),
\end{gather}
The semiclassical limit is the large $k$ limit. 

Let us precise that $\quant_k \rightarrow \J$ is not a genuine vector bundle.
But this can be solved in the semiclassical limit \cite{FoUr}. Indeed,  for any compact submanifold $N$ of $\J$, there exists $k_0(N)$ such that the restriction of $\quant_k$ to $N$ is a vector bundle when $k \geqslant k_0 (N)$. Consequently, $\Psi_k ( \varphi)$ is well-defined only when $k \geqslant k_0 ( \varphi)$. However for this introduction, we will keep our simplified version.

\subsection{Two invariants of symplectic geometry}

Our main result regarding the applications $\Psi_k$ connect them with two
invariants of symplectic topology, the Calabi morphism \cite{Ca} and a
quasimorphism introduced by Shelukhin \cite{Sh}. The Calabi morphism is
usually defined for Hamiltonian diffeomorphims of an open symplectic manifold.
Here, we need a similar morphism for automorphisms of a prequantum bundle over
a compact manifold.

By the Kostant-Souriau prequantization theory \cite{Kos} \cite{So}, the Lie algebra of prequantum bundle
infinitesimal automorphisms is naturally isomorphic with the Poisson algebra of $M$. 
Let $\pathspace ( \preq)$ be the set of smooth
paths of $\preq$ starting from the identity. To any such path $(\ga_t, \; t\in [0,1])$, we
associate a path $(H_t )$ of $\Ci (M, \R)$ representing the derivative $(\dot{\ga}_t)$ through the
Kostant-Souriau isomorphism, that we call the generating Hamiltonian of $\ga$. 

For any path $\ga \in \pathspace ( \preq)$ with generating Hamiltonian $(H_t)$,  we set
\begin{gather} \label{eq:def_cal_intro}
\ca ( \ga ) = \int_0^1 dt \; \int_M H_t (x) d \mu (x)
\end{gather}
where $\mu = \om^n/ n!$ is the Liouville measure.
The map $\ca$  factorizes to a morphism from the universal cover $\widetilde{\preq}^0$
of the identity component $\preq^0$
of $\preq$, into $\R$, already considered in \cite{Sh2}.

The definition of the Shelukhin quasimorphism is more involved and will be postponed to section \ref{sec:space-compl-struct}. Let us discuss its main properties. It is a map
\begin{gather}\label{eq:def_shel_intro}
\Sh_j : \widetilde{\ham} \rightarrow \R
\end{gather}
defined for any symplectic compact manifold $M$ equipped with an almost
complex structure $j$. Here $\widetilde \ham$ is the universal cover of the
Hamiltonian diffeomorphism group of $M$. $\Sh_j$ is a quasi-morphism, i.e.
\begin{gather} \label{eq:quasi-morphism_equation}
  |\Sh_j ( \al \be) -\Sh_j (\al) - \Sh_j ( \be) | \leqslant C
\end{gather}
for a constant $C$ not depending on $\al$, $\be$. This condition is
meaningful because the group $\widetilde{\ham}$ being perfect \cite{Ba}, there
exist no non trivial morphism from $\widetilde \ham$ to $\R$. $\Sh_j$ is non
trivial in the sense that its homogeneisation $\overline{\Sh} ( \al )
:=\lim_{\ell \rightarrow \infty} \Sh_j( \al^\ell)/ \ell$ is not zero. This
homogeneisation $\overline \Sh$ is itself a quasimorphism, it does not depend on
$j$ and it had been defined before for specific classes of symplectic manifolds in \cite{En} and \cite{Py}.  As a last comment,  the construction of $\Sh_j$ is soft in the sense that it does not use pseudo-holomorphic curves. To the contrary, important quasimorphisms of $\widetilde \ham$ can be obtained from Floer theory \cite{PoRo}.  For a general introduction to quasimorphisms in symplectic topology, we refer the reader to \cite{PoRo}.

\subsection{Semiclassical results}  \label{sec:semicl-results}

Recall the well-known fact that the Hamiltonian diffeomorphisms of $M$ are
precisely the diffeomorphisms that can be lifted to a prequantum bundle
automorphism isotopic to the identity.  Furthermore, $M$ being connected, the lift is unique up to a constant rotation, so we have a central extension 
\begin{gather} \label{eq:central_extension_intro}
1 \rightarrow \op{U}(1) \rightarrow \preq^0 \xrightarrow{ \pi}  \ham
\rightarrow 1 ,
\end{gather}
where $\preq^0$ is the identity component of $\preq$.
So  any path $\ga \in \pathspace ( \preq)$ defines a path $\pi ( \ga) \in \pathspace ( \ham)$ and consequently a class $[\pi ( \ga)]$ in $\widetilde \ham$.

\begin{theo} \label{theo:semi-class-result}
  Assume $j_0$ is integrable. Then for any path $\ga \in \pathspace ( \preq )$, the variation $v_k ( \ga)$ of the argument of $t \rightarrow \det \Psi_k ( \ga_t) $ is equal to
  $$ v_k ( \ga) =  - \Bigl( \frac{k}{2 \pi} \Bigr)^n \Bigl( \bigl( k + \lambda' \bigr) \ca ( \ga ) + \tfrac{1}{2} \Sh_{j_0} ( [\pi ( \ga)] ) + \bigo (k^{-1} ) \Bigr)  $$
where $2n$ is the dimension of $M$ and $\lambda' = \frac{n}{2} [c_1 ( M) \cup c_1 (L)^{n-1}] / [ c_1(L)^n \bigr]  $
\end{theo}

For the proof, we use several remarkable results: on one hand by \cite{FoUr},  the curvature
of $\quant_k \rightarrow \J $ is given at first order by the scalar curvature;  on the other hand the definition of $\Sh_j$ is based on the action of the group $\ham$ of Hamiltonian diffeomorphisms on $\J$, action which is  Hamiltonian with a momentum given by the scalar curvature by \cite{Do}.

As a corollary, we will deduce that the lift of $\Psi_k$ to the universal
covers 
\begin{gather} \label{eq:lift_Psi_k}
\widetilde{ \Psi}_k : \widetilde{\preq}^0  \rightarrow \widetilde{\op{U}} (
\Hilb_k)
\end{gather}
is asymptotically a morphism. 
Introduce the geodesic distance $\tilde{d}$ of $\widetilde{\op{U}}(
\Hilb_k)$ corresponding to the operator norm. This distance is controlled at
large scale  by the lift of the determinant $\widetilde{\det} : \widetilde {\op{U}} ( \Hilb) \rightarrow \R$. More precisely, if the dimension of $\Hilb$ is $N$, then for any $\tilde{u}$, $\tilde{v}$ in $\widetilde{\op{U}} ( \Hilb)$, we have
\begin{gather} \label{eq:est_dits_geod}
\frac{ \bigl| \widetilde{\det} \tilde{u} - \widetilde{\det} \tilde{v}  \bigr|
}{N} \leqslant \tilde{d} ( \tilde{u} , \tilde{v} ) \leqslant   \frac{ \bigl|
  \widetilde{\det} \tilde{u} - \widetilde{\det} \tilde{v}  \bigr| }{N}   + 2
\pi
\end{gather}
So the estimate of the argument variation in Theorem
\ref{theo:semi-class-result} and the fact that $\ca$ is a mophism and $\Sh_j$ a
quasimorphism will show the following fact.

\begin{cor} \label{cor:quasimorphism_prod}
  There exists $C>0$ such that for any $\ga_1, \ga_2 \in \wt{\preq}^0$,
\begin{gather} \label{eq:quasimor}
  \tilde{d} ( \widetilde{\Psi}_k ( \ga_1) \widetilde{\Psi}_k ( \ga_2)  ,
  \widetilde{\Psi}_k ( \ga_1 \ga_2) ) \leqslant C + \bigo (k^{-1} ) 
\end{gather}   
with a $\bigo$ depending on $\ga_1, \ga_2$. 
\end{cor}

It is also interesting to compare the map $\Psi_k$ with the quantum propagator
defined through Toeplitz quantization.  For any $f \in \Ci (M, \R)$, we let $T_k (f)$ be the endomorphism of $\Hilb_k$ such that
$  \langle T_k (f) \psi, \psi' \rangle = \langle f \psi, \psi' \rangle$
for any
$\psi, \psi' \in \Hilb_k$.
Define the map
\begin{gather} \label{eq:def_phi_k_intro}
\widetilde \Phi_{k} : \pathspace ( \preq ) \rightarrow \widetilde{\op{U}}(\Hilb_k) 
\end{gather}
as follows. Let $\ga \in \pathspace ( \preq )$ with generating Hamiltonian $(H_t)$. Solve the Schr\"odinger equation
\begin{gather}
  \label{eq:schrodinger}
  U_{t}' = \frac{k}{i} T_k ( H_t) U_t , \qquad U_0 = \op{id}_{\Hilb_k}. 
\end{gather} 
where $(U_t) \in \Ci ( [0,1],  \op{U}(\Hilb_k) )$. Lift the path $(U_t)$ to a
path $(\wt U_t) $ of  $\widetilde{\op{U}}(\Hilb_k)$ starting at the identity
element. And set $\wt \Phi_k ( \ga) := \wt U_1$.

If we have two families $(g_k, h_k \in \wt {\op{U}} ( \Hilb_k), \; k \in \N)$,
we write $g_k = h_k + \bigo ( r_k)$ to say that $ \tilde{d} ( g_k ,
h_k ) = \bigo ( r_k)$. 
\begin{theo} \label{theo:intro2}
    For any path $\ga$ in $\pathspace( \preq)$, $\widetilde \Phi_k ( \ga) = \widetilde{\Psi}_k ( [\ga]) + \bigo (1)$.
\end{theo}

As a consequence if $\ga$, $\ga' \in \pathspace ( \ga)$ are homotopic with fixed
endpoints, then
\begin{gather} \label{eq:homotop}
\widetilde \Phi_k ( \ga) = \widetilde \Phi_k ( \ga')  + \bigo (1) .
\end{gather}  
Furthermore by Corollary \ref{cor:quasimorphism_prod}, 
\begin{gather} \label{eq:quasimor2} 
\widetilde{\Phi}_k ( \ga_1 \ga_2) = \widetilde{\Phi}_k ( \ga_1) \widetilde{\Phi}_k ( \ga_2)  + \bigo (
1)
\end{gather}
Actually it could be possible to deduce \eqref{eq:homotop} and
\eqref{eq:quasimor2} directly from the commutator estimate $\bigl[ T_k (f) , T_k ( g) \bigr] = (ik)^{-1} T_k ( \{ f,g \} ) + \bigo (
k^{-2}) $.

The slight difference between $\wt{\Psi}_k$ and  $\wt{\Phi}_k$ is that the $\bigo(1)$ in
\eqref{eq:quasimor2} depend on $\ga_1$, $\ga_2$, whereas in
\eqref{eq:quasimor} only the $\bigo(k^{-1})$ depends on $\ga_1$, $\ga_2$.
Actually, \eqref{eq:homotop} and \eqref{eq:quasimor2} still hold if we modify
$\widetilde{\Phi}_k$ by a $\bigo (1)$. For instance, we can use any
quantization $T'_k$ such that $T'_k = T_k + \bigo (k^{-1})$. Or we can define
$\widetilde{\Phi}_k$ through the $L^2$-connection by using arbitrary paths in
$\J$.  To the contrary, we can not modify  $\widetilde{\Psi}_k$ arbitrarily by
a $\bigo (1)$ and still having Corollary \ref{cor:quasimorphism_prod}.   So we can view $\widetilde \Psi_k$ as a specific choice amongst all the maps in $\widetilde{\Phi}_k + \bigo (1)$, such that  \eqref{eq:quasimor} holds. 

As a last remark, observe that by composing the maps $\wt \Phi_k$, $\wt
\Psi_k$ with the projection $\wt{\op{U}} ( \Hilb_k) \rightarrow \op{U} (
\Hilb_k)$, we do not obtain anything interesting because the diameter of
$\op{U}(N)$ for the geodesic distance associated to the uniform norm,  is
bounded independly on $N$. So any family in $\op{U} ( \Hilb_k)$ is in $\bigo
(1)$.  
\subsection{Structure of the article}

Because our results are essentially on the quantization of prequantum bundle
automorphisms, Section \ref{sec:preq-bundle-autom} will be devoted to the group $\preq$ and the universal cover of its identity component. We will prove that
$$ \widetilde{\preq}^0 \simeq \R \times \widetilde{ \ham}$$
the isomorphism being the product of the Calabi morphism and the lift of the projection $ \preq^0 \rightarrow \ham$. We will also view that the central extension (\ref{eq:central_extension_intro}) is essentially controlled by the Weinstein action morphism $\pi_1 ( \ham) \rightarrow \op{U}(1)$.  

Section \ref{sec:geod-dist-induc} is devoted to the geodesic distance of the
universal cover of the unitary group induced by the uniform norm. We will
prove estimate \eqref{eq:est_dits_geod},  compute explicitly the
distance, and show that the distance between the identity and a point is
always achieved by a one-parameter semi-group. This does not follow from
general result in Finsler geometry because the uniform norm is not
sufficiently regular. Our proof is actually based on a theorem by Thompson \cite{Th},
which follows itself from Horn conjecture.  

In Section \ref{sec:space-compl-struct}, we give more details on the
definition of the maps $\Psi_k$. Theorem \ref{theo:semi-class-result} is
proved in Section \ref{sec:proof-theor-1}.
Corollary \ref{cor:quasimorphism_prod} and Theorem \ref{theo:intro2} are
proved in Section \ref{sec:proof-theor-2}.

\section{Prequantum bundle automorphisms} \label{sec:preq-bundle-autom}

We study the geometry of the central extension
\eqref{eq:central_extension_intro}. In a first subsection, we introduce a similar finite
dimensional extension with genuine Lie groups, which despite of its simplicity,
already has the main features.
These extensions are also relevant because they appear in our setting when we restrict the Hamiltonian
diffeomorphisms group to the subgroup of isometry for a given metric. The case
of projective manifolds and more generally toric manifolds has been studied in
the literature, as will be explained in the second subsection. In the third subsection, we prove more specific result on the Calabi morphism and the universal covers of
$\preq^0$ and $\ham$. In the last subsection, we explain the relation with the
usual Calabi morphism.

\subsection{A finite dimensional model} \label{sec:finite-dimens-model}

Consider a central extension $G$ of a Lie group $H$ by $\op{U} (1)$. In other words we have an exact sequence of Lie group morphisms
\begin{gather} \label{eq:extension}
  1 \rightarrow \op{U} (1) \rightarrow G \xrightarrow{\pi} H \rightarrow 1.
\end{gather}
such that $\op{U}(1)$ is sent in the center of $G$. 
We assume as well that $H$ and $G$ are connected and that the corresponding exact sequence of Lie algebras
$0 \rightarrow \R \rightarrow \mathfrak{g} \rightarrow \mathfrak{h} \rightarrow 0$ splits. So  $\mathfrak{g} \simeq \R \oplus \mathfrak{h}$ with a Lie bracket of the form
$$[(s, \xi) , ( t , \eta) ]_{\mathfrak{g}} = ( 0, [ \xi, \eta ]_{\mathfrak{h}} ).$$  

Assume for a moment that there exists a group morphism $\si : H \rightarrow G$ integrating the Lie algebra morphism $\mathfrak{h } \rightarrow \mathfrak{g}$, $\xi \rightarrow (0, \xi)$. Then (\ref{eq:extension}) splits. Indeed, $\si \circ \pi = \op{id}_H$ because the derivative of $\si \circ \pi$ is the identity of $\mathfrak{h}$ and $H$ is connected. This implies that $\op{U} (1) \times H \rightarrow G$, $(\theta, g)\rightarrow \theta \si (g)$ is an isomorphism with inverse the map sending $g$ into $ (g \si (\pi(g))^{-1}, \si ( \pi ( g)))$.  

In general, we can define a group morphism $\widetilde{\si}$ from the universal cover $\widetilde{H}$ of $H$ to $G$ and the morphism $\si$ exists only when $\widetilde \si ( \pi_1 (H)) = \{ 1_G \}$. If $\si$ does not exist, we can introduce the subgroup  $K :=  \widetilde \si ( \widetilde H )$ of $G$ as a replacement of $\si (H)$. Let $A: \pi_1 ( H) \rightarrow \op{U} (1)$ be the morphism sending $\ga$ into $\widetilde{\si} ( \ga)$. Then one easily checks that 
$$ 1 \rightarrow \ker A \rightarrow \widetilde H \xrightarrow{\widetilde{\si}} K  \rightarrow 1 ,$$
so $\widetilde H$ is the universal cover of $K$ and $\pi_1 ( K) \simeq \ker A$. Furthermore,
$$ 1 \rightarrow \op{Im} A \rightarrow K \xrightarrow{\pi} H \rightarrow 1 $$
so $K$ is a central extension of $H$ by $\op{Im} A$.

The typical example is $G = \op{U}( n)$ with its subgroup $\op{U}(1)$ of
diagonal matrices so that $H = \op{U} ( n) / \op{U}(1) = \op{PU} ( n)$. Since
the projection $\op{SU} (n) \rightarrow \op{PU} ( n)$ is the universal cover,
we get an identification between the Lie algebras of $\op{PU} (n)$ and $\op{SU}(n)$. We have
$\mathfrak{u} (n) = \R \oplus \mathfrak{su} ( n)$ as required and the group
morphism $\widetilde{\si} : \op{SU}(n) \rightarrow \op{U}(n)$ is merely the
inclusion. So $A$ is the embedding $\Z/n \Z \hookrightarrow \op{U} (1)$ and $K = \op{SU}(n)$.

Another thing that can be done in general is to introduce the  isomorphism of the universal covers
\begin{gather} \label{eq:is_universal_cover_GH}
  \widetilde G \simeq \R \times \widetilde H
\end{gather}
corresponding to the isomorphism of Lie algebras $\mathfrak{g} \simeq \R \oplus \mathfrak{h}$. Observe that the morphisms $\widetilde{\si}$ and $A$ can be recovered from (\ref{eq:is_universal_cover_GH}). Indeed $\widetilde \si$ is the composition of $\widetilde H \rightarrow \widetilde G$, $h \rightarrow (0,h)$ with the projection $\widetilde G \rightarrow G$. 

\subsection{Diffeomorphisms group}

Let $(M, \om)$ be a connected compact symplectic manifold and $\Ham$ be its group of Hamiltonian diffeomorphism. Assume $M$ is equipped with a prequantum bundle $P \rightarrow M$ and let $\preq$ be the the group of prequantum bundle automorphisms of $P$. The precise definition will be given in Section \ref{sec:widetilde-preq}. For now recall the exact sequence of groups
\begin{gather} \label{eq:1}
1 \rightarrow \op{U}(1) \rightarrow \preq^0 \xrightarrow{ \pi}  \ham
\rightarrow 1 ,
\end{gather}
where the projection $\pi$ sends a prequantum bundle automorphism $\varphi$ of $P$ to the diffeomorphism of $M$ lifted by $\varphi$. Furthermore, it is a well-known fact due to Kostant and Souriau that the Lie algebra of infinitesimal prequantum bundle automorphisms of $P$ is isomorphic with $\Ci ( M , \R)$, the Lie bracket being sent to the Poisson bracket. The Lie algebra exact sequence corresponding to (\ref{eq:1}) is
$$ 0 \rightarrow \R \rightarrow \Ci ( M , \R) \rightarrow \Ci ( M, \R) / \R \rightarrow 0 .$$ 
It has a natural splitting $$ \Ci ( M, \R) / \R \simeq \Ci_{0} (M, \R) $$ where $\Ci_0  (M, \R)$ is the subalgebra of $\Ci (M, \R)$ consisting of the functions $f$ having a null average with respect to the Liouville measure. 

So we are exactly in the situation described in Section \ref{sec:finite-dimens-model} except that $\preq$ and $\ham$ are infinite dimensional Lie groups. These groups do not have all the good properties of Lie groups, notably they may have elements arbitrarily closed to the identity and not belonging to any one-parameter subgroup, cf. the remarkable general introduction \cite{Mi} and  \cite{PoSh} for results specific to $\ham$. However, the constructions presented in Section \ref{sec:finite-dimens-model} can be extended to our situation. In particular, we have an isomorphism
\begin{gather} \label{eq:iso_preq_ham}
  \widetilde{\preq}^0 \simeq \R \times \widetilde{ \ham}
\end{gather}
Here the universal covers will be very concretely defined as quotients of diffeomorphism path groups, and the isomorphism will be given by integrating the generating vector fields. The details will be given in the next section.   

The morphism $\op{Cal}: \widetilde{\preq}^0 \rightarrow \R$ given by the
projection on the first factor in (\ref{eq:iso_preq_ham}) will be called the
{\em Calabi morphism} and has already been considered in \cite{Sh2}. Its relation with the usual Calabi morphism will be explained
in Section \ref{sec:usual-calabi-morph}. 

The obstruction $A : \pi_1 (\ham) \rightarrow \op{U}(1)$ to the splitting of
(\ref{eq:1}) can be defined from (\ref{eq:iso_preq_ham}) as explained at the
end of Section \ref{sec:finite-dimens-model}. This morphism $A$ is actually
known in the symplectic topology literature as the {\em Weinstein action
  homomorphism} and was introduced in \cite{We}. 

When $M$ has a complex structure $j$ compatible with $\om$, so that $M$ is K\"ahler, we can also introduce the subgroup $H$ of $\ham$ consisting of the holomorphic Hamiltonian diffeomorphisms, and the subgroup $G$ of $\preq$ consisting of the automorphisms lifting an element of $H$. $G$ and $H$  are genuine Lie groups and satisfy all the assumptions of Section \ref{sec:finite-dimens-model}. The corresponding morphism $A_j : \pi_1 ( H) \rightarrow \op{U} (1)$ is the composition of $A : \pi_1 ( \ham) \rightarrow \op{U}(1)$ with the map $\pi_1 (H) \rightarrow  \pi_1 (\ham)$ induced by the inclusion $H \subset \ham$.

For instance, if $M$ is the projective space $\C \mathbb{P} (n)$ with its
standard symplectic, complex and prequantum structures, then we recover the
example  discussed in Section \ref{sec:finite-dimens-model} where $G
=\op{U}(n+1)$, $H = \op{PU} ( n+1)$. We deduce that $\Z/(n+1)\Z $ embeds into
$\pi_1 ( \ham)$, a well-known fact. More generally, the morphisms $A_j$ is
discussed in \cite{McDTo} for toric manifolds. It is proved that in most
cases, $A_j$ is injective and its image is not finite, \cite[Corollary 2.4 and
Proposition 2.5]{McDu}. So in all these cases, $\pi_1 (H) \subset \pi_1 ( \ham)$ and the image of $A$ is not finite.

\subsection{$\widetilde \preq ^0$, $\widetilde \ham $ and the Calabi morphism} \label{sec:widetilde-preq}

Consider a symplectic compact connected manifold $(M, \om)$. 
Our sign convention for the Hamiltonian vector field $X$ of $f \in \Ci (M, \R)$ and the Poisson bracket are
$$ \om (X, \cdot) + df = 0 , \qquad \{ f,g  \} = X.g .$$
Let $\mu = \om^n/n!$ be the Liouville volume form. A Hamiltonian $f \in \Ci (M, \R)$ is {\em normalised} if $\int_M f \; \mu =0$.  

Let $ \pathspace ( \ham )$ be the group of smooth paths of $\ham$
based at the identity element, the law group being the pointwise product.
Associating each path $(\phi_t)$ of $ \ham $ to its generating vector field $X_t$,
$$ \frac{d}{dt} \phi_t (x) = X_t ( \phi_t (x)), \qquad x \in M, \; t \in [0,1] $$
we obtain a one-to-one correspondence between  $ \pathspace ( \ham)$ and the space of time-dependent Hamiltonians $f \in \Ci ( [0,1] \times M , \R)$ which are normalised at each time $t$. 
As customary in symplectic topology, the group $\widetilde \ham $ is defined as the quotient of $\pathspace ( \ham)$ by the relation of being smoothly homotopic with fixed endpoints, cf. \cite{PoRo} and \cite{Babook}.

Assume now that $M$ is equipped with a prequantum bundle $P \rightarrow M$, that is a $\op{U}(1)$-principal bundle over $M$ endowed with a connection form $\al \in \Om^{1} ( P , \R)$ such that $d \al + \pi^* \om = 0$. Here we identify the Lie algebra of $\op{U} (1)$ with $ \Herm (1) = \R$. So if $\partial_{\theta}$ is the infinitesimal generator of the $\op{U}(1)$-action corresponding to $1$, we have that $\al ( \partial_{\theta} ) = 1$ and $\mathcal{L}_{\partial_{\theta}} \al = 0$. 

An {\em infinitesimal automorphism} of $P$ is a vector field of $P$ commuting with the $\op{U}(1)$-action and preserving $\al$. Any such vector field $Y$ has the form 
\begin{gather} \label{eq:KS}
Y= X^{\op{hor}} - (\pi^* f) \partial_{\theta}
\end{gather}
where $f \in \Ci ( M ,\R)$, $X$ is the corresponding Hamiltonian vector field  and $X^{\op{hor}}$ is the lift of $X$ such that $\al ( X^{\op{hor}}) = 0$. The map sending $Y$ to $f$ is an isomorphism from the space of prequantum infinitesimal automorphims of $P$  to $\Ci ( M , \R)$. The Lie bracket is sent to the Poisson bracket.

The {\em prequantum bundle automorphisms} of $P$ are by definition the diffeomorphisms of $P$ preserving $\al$ and commuting with the $\op{U}(1)$-action. The identity component $\preq^0$ is a central extension of $\ham $ by $\op{U}(1)$, cf. (\ref{eq:1}). The embedding $\op{U}(1) \hookrightarrow \preq$ is given by the action of the structure group of $P$. 
The proofs of the previous claims starting from (\ref{eq:KS}) may be found in \cite[Section 7.1]{BaWe}.

As for the Hamiltonian diffeomorphisms, let $\pathspace ( \preq)$ be the group of smooth paths of prequantum bundle automorphisms based at the identity. To any $(\ga_t)$ in $\pathspace ( \preq)$, we associate its generating vector field $(Y_t)$ and the corresponding time-dependent Hamiltonian $(f_t)$ through (\ref{eq:KS}). This defines a bijection between $\pathspace ( \preq)$ and $\Ci ( [0,1] \times M, \R)$.

We define $\widetilde{\preq}^0$ as the quotient of $\pathspace ( \preq)$ by the relation of being smoothly homotopic with fixed endpoints.  To handle the difference between $\widetilde {\preq}^0$ and $\widetilde {\ham}$, we will need the following Lemma.

\begin{lemma} \label{lem:condition_lift}
Let $(\ga_t^s, (t,s) \in [0,1]^2)$ be a smooth family of $\preq $ such that for any $s\in [0,1]$, $\ga_0^s$ is the identity of $P$ and $\ga_1^s$ lifts the identity of $M$. For any $s \in [0,1]$, set 
$$ C(s) = \int_0^1 \Bigl( \int_M f_t^s \mu \Bigr) dt $$
where $ (f_t^s, t \in [0,1])$ is the Hamiltonian generating $( \ga_t^s, t \in [0,1])$ . Then $s \rightarrow C(s)$ is constant if and only if $s \rightarrow \ga_1^s$ is constant.
\end{lemma}

\begin{proof}  Let $(g_t^s)$ be the smooth family of $\Ci ( M ,\R)$ such that for any $t \in [0,1]$, $(g^s_t , s \in [0,1])$ is the Hamiltonian generating $(\ga_t^s, s \in [0,1])$. Then 
\begin{enumerate} 
\item $g_0 ^s = 0$ because $\ga_0^s$ is the identity.
\item $g_1^s =: D(s) \in \R$ because $\ga_1^s$ lifts the identity of $M$ 
\item $ \partial f_t^s / \partial s - \partial g_t^s / \partial t = \bigl\{
  f_t^s, g_t^s \bigr\}$, as a consequence of the differential homotopy
  formula, cf . \cite{Ba} Proposition I.1.1.
\end{enumerate}
Using that $\int_M \{ f_t^s, g_t^s \} \mu =0$, we obtain  
\begin{xalignat*}{2} 
C'(s) &  = \int_0^1 \int_M  \frac{\partial f_t^s}{\partial s}   \mu \; dt = \int_0^1 \int_M   \frac{\partial g_t^s}{\partial t} \mu \; dt =  \int_0^1  \frac{\partial}{\partial t} \Bigl( \int_M   g_t^s \mu \Bigr) dt \\ 
 & = \int_M g_1^s \mu - \int_M g_0^s \mu = \op{Vol} (M) D(s)  
\end{xalignat*}
which concludes the proof.
\end{proof}

Introduce the group morphism
$$R : \R \rightarrow \wt \preq ^0$$
lifting the embedding of $U(1)$ into $\preq$. More explicitly, $R( \tau)$ is the class of the path $t\in [0,1] \rightarrow e^{it \tau} $. The image of $R$ is contained in the center of $\wt \preq ^0$. 

Define the map $\ca$ from $\pathspace (\preq ) $ to $\R$ by 
$$\ca (\ga_t) = \int_0^1 \Bigl( \int_M f_t \mu \Bigr) dt$$
where $(f_t)$ is the Hamiltonian generating $(\ga_t)$.
As we will see, $\ca $ factorizes to $\wt \preq ^0$ which defines
our {\em Calabi morphism}.

\begin{prop} \label{prop:Calabi_morphism}
$  $
  
  \begin{enumerate}
  \item $\ca $ is a group morphism, which factorizes to a  morphism $\ca $
    from $ \preq ^0$ to $ \R$.
    \item for any $\tau$, $ \ca ( R(\tau)) = - \tau$. 
    \item $\ca ( \ga) = 0$ if and only if $\ga$ has a representative whose
      generating Hamiltonian $(f_t)$ is normalised, that is $\int_M f_t \mu =
      0$, for every $t
      \in [0,1]$.
    \end{enumerate}
\end{prop}

\begin{proof}
  To check that  $\ca  (\ga_t \ga'_t ) = \ca ( \ga_t) +  \ca  ( \ga'_t)$, one first compute the generating Hamiltonian of $(\ga_t \ga'_t)$ in terms of the generating Hamiltonian of $(\ga_t)$ and $(\ga'_t)$:
  \begin{gather} \label{eq:prod_path}
 ( f_t ) \star (f'_t) = ( f_t + f'_t \circ \al_t^{-1})
\end{gather}
where $\al_t$ is the Hamiltonian flow of $(f_t)$. The same proof for the usual Calabi morphism is presented in \cite[Theorem 4.1.1]{PoRo}. 
 The fact that $ \ca  (
\ga_t) = \ca ( \ga'_t)$ when $(\ga_t)$ and $(\ga_t)$ are smoothly
homotopic with fixed endpoint, follows from Lemma \ref{lem:condition_lift}.

The second property $\ca  (R( \tau) ) = -\tau$ is straightforward.

Let $(\ga_t) \in \pathspace ( \preq )$ with generating Hamiltonian $(f_t)$. Let $\theta
: [0,1] \rightarrow \R$. Then the generating Hamiltonian of $(e^{i \theta (t)}
\ga_t)$ is $f_t - \theta'(t)$. This Hamiltonian is normalized if we define
$\theta$ by 
$$\theta (t) = \int_0^t \Bigl( \int_M f_t \mu \Bigr) dt .$$
Furthermore, when $\ca ( \ga_t) =0$, we have $\theta( 0 ) = \theta (1) = 0$, so
that $(\ga_t)$ and $(e^{i \theta (t)} \ga_t )$ are smoothly homotopic with
fixed endpoints. 
\end{proof}

The projection
$ \pi: \preq^0 \rightarrow \ham$ induces a group morphism
$$ \pi: \wt \preq^0 \rightarrow \wt \ham .$$
We can also define a left inverse
$$ L : \wt \ham \rightarrow \wt \preq ^0$$
by sending the class of a path $(\varphi_t)$ to the class of the path $(
\ga_t)$ lifting $(\varphi_t)$ and having a normalized generating Hamiltonian.
\begin{prop}
  $ $
  \begin{enumerate} 
  \item   $L$ is well-defined, it is a group morphism,

    \item $\pi \circ L $ is the identity
  of $\wt \ham$,  the image of $L$ is the kernel of the Calabi morphism,
\item the kernel of
    $\pi$ is the image of $R$.
  \end{enumerate}
\end{prop}

\begin{proof}
The path $(\ga_t)$ certainly exists because it is the flow of the vector field
associated to the normalized generating Hamiltonian of $(\varphi_t)$. The fact
that the class of $(\ga_t)$ only depends on the class of $(\varphi_t)$ follows
from Lemma \ref{lem:condition_lift}. $L$ is a group morphism because   the
product of normalized Hamiltonians corresponding to the product of
$\pathspace(\ham)$ is also given by formula (\ref{eq:prod_path}).

$\pi \circ L = \op{id}_{\wt \ham}$ is obvious, the assertion on the image of
$L$ is  the third assertion of Proposition \ref{prop:Calabi_morphism}.

For the last point, the image of $R$ is certainly contained in the kernel of
$\pi$. To show the converse,  consider $(\ga_t) \in \pathspace ( \preq)$ with generating
Hamiltonian $(f_t)$. Write $f_t = \theta( t) + g_t$ where for any $t$,
$\theta(t) \in \R$ and $g_t$ is normalized. Then $[\ga_t] =R(\tau) L (
\varphi_t)$ with  $\tau = \int_0^1 \theta(t) dt$ and $(\varphi_t)$  the path
of $\ham$ generated by $(g_t)$.  Now $\pi ( \ga_t) = 0$ implies that
$[\varphi_t ] =0$ because $\pi \circ L = \op{id}_{\wt \ham}$.  
\end{proof}

We deduce that the groups $\wt \preq $ and $\wt \ham \times \R$ are
isomorphic. 

\begin{cor} The group isomorphisms
  \begin{xalignat*}{2}
    \wt \preq ^0 & \rightarrow \wt \ham \times \R & \wt \ham \times \R &
    \rightarrow \wt \preq ^0 \\ 
    g & \rightarrow (\pi (g), \ca ( g)) , & ( h, \tau) & \rightarrow L(h) R( \tau)
  \end{xalignat*}
  are inverse of each other. 
\end{cor}

\subsection{The usual Calabi morphism}\label{sec:usual-calabi-morph}

Let $x \in M$. Introduce the subgroup $\ham_x$ of $\ham$ consisting of the Hamiltonian diffeomorphisms fixing $x$, and the subgroup $\preq_x$ of $\preq$ consisting of the prequantum bundle automorphisms fixing the fiber $P_x$.
We claim that the morphism $\preq_x \rightarrow \ham_x$, $\gamma \rightarrow \pi(\ga)$ is an isomorphism. The injectivity follows directly from the exact sequence (\ref{eq:1}). The surjectivity is a consequence of the connectedness of $\ham_x $. This can be proved by using the long exact sequence for homotopy groups associated to the fibration
$\ham \rightarrow M$, $\phi \rightarrow \phi (x)$, and the fact that the
morphism $\pi_1 ( \ham) \rightarrow \pi_1 (M)$ is trivial, a folk Theorem
according to \cite[Footnote 3]{McDu}.

Since $\ham_x$ and $\preq_x$ are isomorphic, the same holds with their universal covers and composing with the Calabi morphism introduced in Proposition \ref{prop:Calabi_morphism}, we obtain a morphism $\widetilde{\ham}_x \rightarrow \R$.

Introduce now the open set $U = M \setminus \{ x\}$ and let $\ham_c( U)$ be the group of compactly supported Hamiltonian diffeomorphisms of $U$. $\ham_c( U)$ identifies with a subgroup of $\ham_x$, so we get a morphism $\widetilde \ham_c( U) \rightarrow \widetilde \ham_x $, which after composition with the previous morphism gives us $\widetilde \ham_c( U) \rightarrow \R$. This morphism is the usual Calabi morphism, defined for instance in \cite[Section 4.1]{PoRo}. Indeed, both are defined by the same formula.

\section{Geodesic distance induced by operator norm} \label{sec:geod-dist-induc}

We start with geodesic distance of the unitary group, the results are certainly
standard but we do not know any reference. The second subsection is devoted to
the universal cover of the unitary group. 
\subsection{The unitary group}
Let $\Hilb$ be a finite dimensional Hilbert space. We denote by $\op{U} ( \Hilb )$ and $\Herm ( \Hilb)$ the spaces of unitary and Hermitian endomorphisms of $\Hilb$ respectively. We consider $\Herm ( \Hilb)$ as the Lie algebra of $\op{U}(\Hilb)$.
Denote by $\| A \|$ the operator norm of any endomorphism $A$ of $\Hilb$. We define for any piecewise $\Cl^1$ curve $\ga : [a,b] \rightarrow \op{U}(\Hilb)$, its length
$$ L ( \ga) = \int_{a}^{b} \| \ga'(t)  \| \; dt $$ 
and the corresponding distance in $\Hilb$ 
$$ d ( u, v) = \inf \{ L ( \ga ) ; \; \ga \text{ is a curve with endpoints } u, v \}.  $$
Since $\| \ga' (t) \| = \| \ga(t)^{-1} \ga' (t) \|$, this distance is the geodesic distance of $\op{U}(\Hilb)$ for the invariant Finsler metric given corresponding to the operator norm. There is a large literature on Finsler geometry, but it does not say anything on $d$ because the operator norm is not enough regular to apply the variational method. So we have to study $d$ from its definition. 

We easily see that $L$ is invariant under reparameterization and that $d$ is symmetric and satisfies the triangle inequality. Since $\| \cdot \|$ is left and right-unitarily invariant, we have for any $g \in \op{U}(\Hilb)$, $  L( \ga)= L(\ga g)= L ( g  \ga)$. Furthermore  $L ( \ga)=  L ( \ga^{-1})$. Consequently
$$ d ( u,v) =  d ( gu , gv) = d ( ug , vg) = d ( u^{-1},v^{-1}). $$
Let us compute explicitly $d(u,v)$, which will prove that $d$ is non-degenerate in the sense that $d(u,v) =0$ only when $u=v$. 
\begin{prop} \label{prop:distance_unitary_group}
For any $u,v \in \op{U}(\Hilb)$, 
\begin{gather} \label{eq:distance_unitary}
d (u, v) = \max  |  \arg \la_i |
\end{gather}
where the $\la_i$'s are the eigenvalues of $u^{-1} v$ and $\arg$ is the inverse of the map from $]-\pi, \pi]$ to $\op{U}(1)$ sending $\theta$ to $e^{i \theta}$. So we have 
\begin{gather} \label{eq:equivalence_distance_unitary}
 \| u - v \| \leqslant d(u,v) \leqslant \frac{\pi}{2} \| u - v \| 
\end{gather}
\end{prop}
As a last remark, observe that the diameter of $\op{U}( \Hilb)$ is $\pi$.

\begin{proof} Since $d(u,v) = d ( 1, u^{-1} v) $, without loss of generality, we may assume that $u=1$.   Working with an orthonormal basis of $v$, we construct $\xi \in \Herm ( \Hilb)$  such that $\exp (i \xi) =  v$ and $\| \xi \| \leqslant \pi$. The length of the curve $ [0,1] \ni t \rightarrow \exp (i  t \xi)$ is $\| \xi \| = \max |\theta_j|$ with $\te_j = \arg \la_j$. So $d(1,v) \leqslant \max |\theta_j|$.

Conversely, choose a normalised eigenvector $X_j$ of $v$ with eigenvalue $e^{i \theta_j}$. 
Consider any curve $\ga : [a, b ] \rightarrow \op{U}(\Hilb)$ from $1$ to $v$. Set $X(t) = \ga(t) X_j$. Since $X(b) = e^{i\theta_j} X(a)$, the geodesic distance in the unit sphere of $\Hilb$ between $X(b)$ and $X(a)$ is $|\te_j|$. So
$$ |\te_j|  \leqslant \int_a^b \|  X'_t  \|   \; dt  \leqslant \int_a^b \| \ga' (t) \| \; dt = L ( \gamma)
$$
So $ | \theta_j| \leqslant d(1,v) $ for any $j$, so $\max |\theta_j| \leqslant d(1,v)$. 

To prove (\ref{eq:equivalence_distance_unitary}), we can again assume that $u=1$. Clearly, $\| v -1 \| = \max |\la_j -1 |$. Then observe that for $\lambda = e^{i \theta}$, $| \la -1 | = 2 | \sin ( \theta /2)|$ so if $\theta \in [-\pi, \pi ]$, we get that $|\la -1 | \leqslant |\theta | \leqslant  \frac{\pi}{2} |\la -1 |$.
\end{proof}

\subsection{Universal cover of $\op{U}(\Hilb)$}
The universal cover $\widetilde{\op{U}} ( \Hilb)$ of the unitary group of $\Hilb$ can be realised as the subgroup of $\op{U}(\Hilb) \times \R$ consisting of the pairs $( u, \varphi)$ such that $\det u = e^{i \varphi}$. Note that the determinant map from $U (\Hilb)$ to $\op{U}(1)$ lifts to the map 
$$ \widetilde{\det} : \widetilde{\op{U}} ( \Hilb) \rightarrow \R, \qquad \widetilde{\det} ( u, \varphi ) = \varphi $$
Endowing $\Herm ( \Hilb)$ with the operator norm, we define the length of the curves $\widetilde{\op{U}} ( \Hilb)$. Then for any $\tilde u, \tilde v \in \widetilde{\op{U}} ( \Hilb)$, we define $\tilde{d} ( \tilde{u}, \tilde{v} )$ as the infimum of the lengths of the curves connecting $\tilde{u}$ and $\tilde{v}$. The situation is exactly the same as for $d$: we easily proved that $\tilde{d}$ is symmetric, satisfies the triangle inequality, is left and right invariant. But it is a priori not clear that $\tilde{d}$ is non degenerate.

Let $\tilde u = ( u, \varphi)$ and $\tilde v = ( v, \psi)$. The curves of $\widetilde{\op{U}} ( \Hilb)$ connecting $\tilde{u}$ and $\tilde{v}$ can be identified with the curves of $\op{U} ( \Hilb)$ connecting $u$ and $v$ and such that the angle variation \footnote{the angle variation of a curve $z:[a,b] \rightarrow \op{U}(1)$ is $-i \int_a^b z'(t)/ z(t) \; dt $. } of their determinant if $\psi - \varphi$. So 
\begin{gather} \label{eq:major_distance} 
  d(u,v) \leqslant  \tilde{d} (\tilde u , \tilde v)
\end{gather}
Furthermore using that for any curve $\ga$ of $\op{U} ( \Hilb)$, the logarithmic derivative of $\det( \ga (t))$ is $ \op{tr} ( \ga^{-1} ( t) \ga' (t) ) $ and that for any Hermitian matrix, $| \op{tr} (A ) | \leqslant N \| A \|$ where $N = \dim \Hilb$, we get
\begin{gather} \label{eq:major_trace}
 \frac{| \psi - \varphi|}{N} \leqslant  \tilde{d} ( \tilde u , \tilde v ) 
\end{gather}
The inequalities (\ref{eq:major_distance}) and (\ref{eq:major_trace}) imply that $\tilde{d}$ is non-degenerate.

We are now going to compute explicitly the distance from  $\tilde u = ( u, \varphi)$ to $\tilde v = ( v, \psi)$. 
Denote by $\la_i$, $i =1, \ldots, N$ the eigenvalues of $v u^{-1}$. Let
$(X_i)$ be an associated orthonormal eigenbasis. Introduce the
$(N-1)$-dimensional affine lattice of $\R^N$ 
$$R = \{ \theta \in \R^N ;  \;  e^{i \theta_i} =  \la_i , \; \forall i \text{ and } \textstyle \sum_{i=1}^N \theta_i = \psi - \varphi \bigr\}$$
For any $\theta \in R$, let $H \in \Herm ( \Hilb)$ be such that $HX_i = i \theta_i X_i$. 
Then $\exp(i H)= v u^{-1}$ and $ \op{tr} H = \psi - \varphi$, so the curve
$$[0,1] \rightarrow \widetilde{\op{U}} ( \Hilb), \qquad t \rightarrow ( e^{ itH} u , \varphi + t \op{tr} H) $$
goes from $(u,\varphi)$ to $( v, \psi)$. Its length is $\| H \| = \max | \theta_i|$. This proves that 
\begin{gather} \label{eq:minoration}
  \tilde{d} ( ( u, \varphi), ( v, \psi)) \leqslant m  \qquad \text{ with }  m
  =  \min  \Bigl\{  \max_{i =1, \ldots,n} |\theta_i| ;\; \theta
  \in R  \Bigr\}  
\end{gather}
By Proposition \ref{prop:distance_unitary_cover}, we actually have an equality.
 So there exists a length minimizing curve $( \ga(t), \varphi (t)) $ connecting $(u,\varphi)$ to $( v, \psi)$ such that $\ga^{-1} ( t) \ga' (t) $ is constant.

We first prove the following partial result, which is actually sufficient for our applications.  

\begin{prop}  \label{prop:estimation_universal_cover}
We have 
$$ \tilde{d} ( \tilde{u}, \tilde {v} ) \leqslant \frac{\pi}{2N} \Rightarrow \tilde{d}( \tilde{u}, \tilde{v} ) = d ( u, v) = m $$
Furthermore
$$  \frac{| \psi - \varphi|}{N} \leqslant  \tilde{d} ( \tilde u , \tilde v ) \leqslant  \frac{| \psi - \varphi|}{N} + 2 \pi $$
\end{prop}

\begin{proof}  Assume that $\tilde{d} ( \tilde{u}, \tilde {v} ) \leqslant \pi / (2N)$. 
  By Proposition \ref{prop:distance_unitary_group}, $d(u,v) = \max | \theta_j|$ where $\theta_j = \arg \la_j$.   By (\ref{eq:major_distance}), $d(u,v) \leqslant \pi /2N$, so $|\theta_j | \leqslant \pi /2N$ and consequently $| \sum \theta_j| \leqslant \pi/2$. By (\ref{eq:major_trace}), $|\psi - \varphi| \leqslant \pi/2$. Since $\psi -\varphi = \sum \theta_j $ modulo $2\pi$, we deduce that $\psi - \varphi= \sum \theta_j$, that is $(\theta_i) \in R$. Since $|\theta_i| \leqslant \pi/2$ for any $i$, we have $d(u,v) = \max | \theta_i| = m$, which proves the first part. 

 The second part  will follows from $m \leqslant  \frac{| \psi - \varphi|}{N} + 2 \pi$. Denote by $d_N$ the distance of $\R^N$ associated to the sup norm, so $m = d_N ( 0 , R)$. Let $\al = \frac{\psi - \varphi}{N} (1,\ldots, 1)$. Then $m \leqslant d_N ( 0, \al) + d_N ( \al, R) = \frac{|\psi - \varphi |}{N} + d_N (0 , R- \al)$. Now $R-\al$ is an affine lattice of the hyperplane $\{ \sum \theta_i = 0 \}$, directed by $\bigoplus_{j=1}^{N-1} \Z f_j$, where $f_j = 2 \pi (e_j - e_{j+1})$, $(e_j)$ being the canonical basis of $\R^N$.  Hence $R-\al$ contains a point $\be = \sum x_i f_j$ with $ |x_j| \leqslant 1/2$. So $d_N ( 0, R - \al ) \leqslant d_N( 0, \be) \leqslant 2 \pi$. 
\end{proof}

\begin{prop} \label{prop:distance_unitary_cover}
  We have $\tilde{d} ( ( u, \varphi), ( v, \psi)) =  m $. 
\end{prop}
The proof is based on Thompson theorem \cite{Th}, which is not quite a elementary result because it follows from Horn conjecture. The idea to apply Thompson theorem comes from the paper \cite{AnLaVa} where the geodesic distance of the unitary group associated to the Schatten norms is computed.

\begin{proof}We prove the result by induction on $k$ where
$$  \tilde{d} ( ( u, \varphi), ( v, \psi)) \leqslant k \frac{\pi}{2N} $$
For $k=1$, this was the first part of Proposition \ref{prop:estimation_universal_cover}. 
Without loss of generality, we may assume that $(u, \varphi) =(1,0)$. 
Assume that 
\begin{gather} \label{eq:hyp_recurrence}
k \frac{\pi}{2N} < \tilde{d} ( ( 1, 0), ( v, \psi)) \leqslant (k+1) \frac{\pi}{2N}
\end{gather} 
Let $S = \{ g \in \widetilde{\op{U}} ( \Hilb)  / \tilde{d} ( (1,0), g)= \pi/2N \}$. Endow $\widetilde{\op{U}} ( \Hilb)$ with the subspace topology of $ \op{U} ( \Hilb) \times \R$. By (\ref{eq:major_distance}) and (\ref{eq:major_trace}), $\tilde d$ is continuous and $S$ is compact. So there exists $( w,\xi)$ in $S$ such that $\tilde{d} ( (w, \xi) , ( v, \psi)) = \tilde{d} ( S, ( v, \psi))$. We claim that 
\begin{gather} \label{eq:deuxmorceaux}
 \tilde{d}((1,0), ( v, \psi) ) = \frac{\pi}{2N} + \tilde{d} (  (w, \xi) , ( v, \psi))
\end{gather}
Indeed, the left-hand side is smaller than the right-hand side by triangle inequality. Conversely, if $g(t)$ is any curve from $(1,0)$ to $( v, \psi)$, by continuity of the function $t \rightarrow \tilde{d} ( (1,0), g(t) )$, $g$ meets $S$ at a point $g(t_0)$. So the length of $\ga$ is larger than $\tilde{d} (( 1, 0), g(t_0) ) + \tilde{d} ( g(t_0) , ( v, \psi)) \geqslant \frac{\pi}{2N} +  \tilde{d} ( S,  ( v, \psi))$, which conclude the proof of (\ref{eq:deuxmorceaux}).

By (\ref{eq:hyp_recurrence}) and (\ref{eq:deuxmorceaux}), we have that $ \tilde{d} (  (w, \xi) , ( v, \psi)) \leqslant k \pi/2N$. Assume that the result is already proved for $  (w, \xi)$, $( v, \psi)$. So  there exists $H \in \Herm ( \Hilb)$ such that 
$$\| H \| = \tilde{d}((w, \xi),( v, \psi)), \quad e^{iH} = v w^{-1}, \quad \op{tr}  H =  \psi - \xi .$$
By the first part of Proposition \ref{prop:estimation_universal_cover}, there exist $H' \in \Herm ( \Hilb) $ such that 
$$ \| H' \| = \frac{\pi}{2N}, \quad e^{iH'} = w, \quad \op{tr} H' =  \xi. $$
So $v = e^{ iH} e^{i H'}$. By Thompson Theorem, there exists $K \in \Herm ( \Hilb ) $ such that $e^{iK } = v $ and $K = UHU^* + V H' V^*$ for two unitary
endomorphism of $\Hilb$. Hence $\op{tr} K = \op{tr} H + \op{tr} H' = \psi$. This
implies by \eqref{eq:minoration} that  $  \tilde{d}((1,0), ( v, \psi) ) \leqslant \| K
\| $. On the other hand, $\| K \| \leqslant \| H \| + \| H' \| =
\tilde{d}((1,0), ( v, \psi) )$ by (\ref{eq:deuxmorceaux}). Hence
$$  \tilde{d}((1,0), ( v, \psi) ) = \| K \|.$$ So the curve $[0,1] \ni t \rightarrow e^{tK}$ is a length minimizing curve from $(1,0)$ to $(v, \psi)$.  
\end{proof}

\section{The map $\Psi_k : \preq \rightarrow \op{U} (
\Hilb_k)$}  \label{sec:space-compl-struct}

\subsection*{Space of complex structures} 

Let $E$ be a finite dimensional symplectic vector space. Let $\J (E)$ be the space of linear complex structures $j$ of $E$ which are compatible with the symplectic form $\om$ of $E$ in the sense that $\om ( jX,jY) = \om (X,Y)$ and $\om ( X, jX) >0$ for any $X \in E \setminus \{ 0 \}$. $\J (E)$ is isomorphic to the Siegel upper half-space $\op{Sp} (2n) / \op{U} (n)$. It has a natural K\"ahler metric defined as follows. First $\J (E)$ is a submanifold of the vector space $\op{End} (E)$ and 
$$ T_{j} \J (E) = \{ a \in \op{End} (E) / \; ja + aj = 0, \; \om ( a \cdot, j \cdot )+ \om ( j \cdot, a \cdot ) = 0 \} $$
The symplectic form of $T_{j} \J (E)$ is given by $\si _{j} ( a,b) = \frac{1}{4} \op{tr} ( jab)$. The complex structure of $T_{j} \J (E)$ is the map sending $a$ to $ja$. The corresponding Riemannian metric of $\J (E)$ has the property that any two points are connected by a unique geodesic. 

Consider now a compact symplectic manifold $M$ and $\J $ be the space of (almost) complex structures of $M$ compatible with the symplectic form. $\J $  may be considered as an infinite dimensional manifold, whose tangent space at $j$ consists of the sections $a$ of $\op{End} TM$ such that at any point $x$ of $M$, $a(x)$ belongs to $T_{j(x)} \J(T_x M)$. Define the symplectic product
\begin{gather} \label{eq:symp_product}
 \si_j  ( a, b) = \int_M \si_{j(x)} ( a(x) , b(x) ) \mu (x), \qquad \forall a,b \in T_j \J  .
\end{gather}
For any $j \in \J $, introduce the Hermitian scalar curvature $S (j) \in \Ci (M, \R)$ defined as follow. The canonical bundle $ \wedge_j ^{ n,0} T^*M$ has a natural connection induced by $j$ and $\om$.  Then $S(j)  \om^n = n \rho \om^{n-1}$, where $\rho$ is $-i$ times the curvature of $ \wedge_j ^{ n,0} T^*M$. 
As was observed by Donaldson in \cite{Do}, the action of $\op{Ham} M$ on $\J
$ is Hamiltonian with momentum $- S(j)$. The sign convention for the momentum is different from ours in \cite{Do}.

\subsection*{Shelukhin quasi-morphism} 
Fix $j_0 \in \J ( M)$. For any $[\varphi_t] \in \widetilde{\Ham} $  with
generating Hamiltonian $(H_t)$, let $j_t = (\varphi_t )_*j_0$. Let $(g_t, t
\in [0,1])$ be the curve of $\J $ joining $j_1$ to $j_0$ such that for any
$x \in M$, $ s \rightarrow g_s (x)$ is the geodesic joining $j_1 ( x)$ to $j_0
( x)$. In the sequel we call $(g_t)$ the {\em geodesic} joining $j_1$ to $j_0$. 
Let $D$ be any disc of $\J $ bounded by the concatenation $(\ell_t)$ of $(j_t)$ and $(g_t)$. Set 
\begin{gather} \label{eq:def_quasi_morphism} 
 \Sh_{j_0} ( [\varphi_t] ) = \int_{D} \si + \int_0^1 \Bigl( \int_M S( J_t) H_t \mu \Bigr) dt 
\end{gather}
More precisely, viewing $D$ as a smooth family of discs $D_x \subset \J(T_xM)$, the first integral is $\int_M \bigl( \int_{D_x} \si_{x} \bigr) \mu (x)$ where $\si_x$ is the symplectic form of $\J(T_xM)$. Observe also that $\J(T_xM)$ being contractible, $\int_{D_x} \si_{x}$ does not depend on the choice of $D_x$. A possible choice for $D_x$ is the map sending $r \exp ( 2 \pi i t)$ to the point with coordinate $r$ of the geodesic joining $j_0 (x)$ to $\ell_t (x)$. 

The map introduced in \cite{Sh} is $n! \Sh _{j_0}$, and the scalar curvature is defined
with the opposite sign. By this paper, $\Sh_{j_0}$ is well-defined, that is the
right-hand side of (\ref{eq:def_quasi_morphism}) only depends on the homotopy
class of $(\varphi_t)$. 
\subsection*{A quantum bundle on $\J $}

Assume now that $M$ is endowed with a prequantum line bundle $L \rightarrow M$, so the connection $\nabla$ of $L$ has curvature $\frac{1}{i} \om$. 
Let $j \in \J $ and $k \in \N$. Then $j$ and $\om$ induces a metric on $\wedge^1 T^*M$, so that we can define the adjoint of $$\nabla^{L^k} : \Ci ( M , L^k) \rightarrow \Om^1 ( M , L^k).$$ 
Let $\Delta_k (j) = ( \nabla^{L^k} )^{*j} \nabla^{L^k}$ be the Laplacian acting  on $\Ci (M, L^k)$. Define $\quant_k (j)$ as the subspace of $\Ci ( M , L^k)$ spanned by the eigenvectors of $\Delta_k (j)$ whose eigenvalue is smaller than $nk + \sqrt k$. We would like to think of $\quant_k (j)$ as the fiber at $j$ of a vector bundle $\quant_k \rightarrow M$. But $\quant_k$ is not a genuine vector bundle, for instance the dimension of $\quant_k ( j)$ depends on $j$. 

Nevertheless, as shown in \cite{FoUr}, \cite{MaMa}, $\Delta_k (j)$ has  the following spectral gap: there exists positive constants $C_1$, $C_2$ independent of $k$ such that the spectrum of $\Delta_k - nk$ is contained in $(-C_1, C_1 ) \cup ( k C_2 , \infty)$. Here, $C_1$ and $C_2$ remain bounded when $j$ runs over a bounded subset of $\J$  in $\Cl^3$-topology. Furthermore, the dimension of the subspace spanned by the eigenvectors with eigenvalue in $(-C_1, C_1)$ is equal to $\int_M \exp (k \om / (2 \pi)) \op{Todd} M $ when $k $ is sufficiently large. 

Now, choose any compact submanifold  $S$ of $\J $, that is a smooth family of $\J $ indexed by a (finite dimensional) compact manifold (possibly with boundary). Then by the above spectral gap, there exists a constant $k(S)$, such that when $k$ is larger than $k(S)$, the restriction of $\quant_k $ to $S$ is a vector bundle. Furthermore $Q_k \rightarrow S$ being a subbundle of the trivial vector bundle $S \times \Ci ( M , L^k)$, it has a natural connection  $ \nabla^{\quant_k} = \Pi_k  \circ d$ where $d$ is the usual derivative and $\Pi_k (j)$ is the orthogonal projection from $\Ci ( M , L^k)$ onto $\quant_k (j)$. 
The main result of \cite{FoUr} (Theorem 2.1) is that the curvature $R_k$ of $\nabla^{\quant_k}$ has the form 
\begin{gather} \label{eq:Foth_uribe_curvature} 
 R_k (a,b) = -\frac{1}{2} \Pi_k \si_j (a,b)  + \bigo ( k^{-1}) , \qquad a,b \in T_j  \J 
\end{gather}
where $\si_j (a,b)$ is the symplectic product defined in (\ref{eq:symp_product}).

\subsection*{The map $\Psi_k : \preq  \rightarrow \op{U} ( \Hilb_k ) $ and its lift}
Consider the group $\preq$ of prequantum bundle automorphism of $L$ acting  on  $\Ci (M, L^k )$ by push-forward. If $\varphi \in \preq $ and
$j \in \J $, then $\varphi_*$ intertwines the Laplacians $\Delta_k (j)$ and
$\Delta_k ( \pi ( \varphi)_* (j))$ where $\pi (\varphi)$ is the Hamiltonian
diffeomorphism of $M$ associated to $P$. So $\varphi_*$ restricts to a
unitary map from $\quant_k (j)$ to $ \quant_k ( \pi ( \varphi)(j))$.

Let $j_0$ be a fixed complex integrable structure. For any $k$, set $\Hilb_k := \quant_k ( j_0)$. 
Let $\varphi \in  \preq $, $j_1 = \pi(\varphi)_* j_0$ and  $g$ be the geodesic segment joining $j_0$ and $j_1$. Assume that $k$ is larger than $k(g)$ and let $\mathcal{T}_k : \quant_k ( j_1) \rightarrow \quant_k ( j_0) $ be the parallel transport from $j_1$ to $j_0$ along $g$.  
 When   $k$ is larger than $k(g)$, we set
 $\Psi_{k} ( \varphi)  := \mathcal T_{k} \circ \varphi_* \in \op{U} ( \Hilb_k ).$
 
We lift $\Psi_{k}$  to a map 
\begin{gather} \label{eq:def_lift_Psi}
\widetilde{\Psi}_{k} : \pathspace ( \preq) \rightarrow \widetilde{\op{U}} ( \Hilb_{k}) 
\end{gather}
in such a way that for any path $\ga = (\ga_t) $, $\widetilde{\Psi}_{k} ( \ga)$ is the endpoint of the lift of $ t \rightarrow \Psi_{k} ( \ga_t) $. More precisely, let $j_t = \pi( \ga_t)_*j_0$ and let $S$ be a surface of $\J ( M)$ containing the geodesic segments joining $j_0$ to $j_t$ for any $t \in [0,1]$. Then when $k$ is larger than $k(S)$, the parallel transport in $Q_k$ along these geodesic segments is well-defined and the resulting map $\Psi_{k} ( \ga_t)$ depends continuously on $t \in [0,1]$.  So it can be lifted to $\widetilde{\op{U}} ( \Hilb_{k}) $.

The map $\widetilde{\Psi}_{k}$ almost factorizes to a map from $\widetilde{\preq} $ to $\widetilde{\op{U}} ( \Hilb_{k})$. Indeed, for any two path $\ga, \ga'$ which are homotopic with fixed endpoints, there exists $k ( \ga, \ga')$ such that for $k \geqslant k (\ga, \ga')$, $\widetilde{\Psi}_{k} ( \ga) = \widetilde{\Psi}_{k} ( \ga')$.

\section{Proof of Theorem \ref{theo:semi-class-result}}  \label{sec:proof-theor-1}

We will prove that for any $\ga \in \pathspace ( \preq )$, we have 
\begin{gather} \label{eq:toprove} \widetilde{\det} \bigl( \widetilde{\Psi}_{k} ( \ga ) \bigr) = -
\Bigl(\frac{k}{2 \pi} \Bigr)^{n} \Bigl( \bigl( k + \lambda' \bigr) \ca ( [\ga]
) + \tfrac{1}{2} \Sh_{j_0} ( [\pi ( \ga)] ) + \bigo ( k^{-1}) \Bigr)
\end{gather}
where the constant $\la'$ is 
\begin{gather} \label{eq:lambda}
\lambda' = \frac{n}{2} \frac{ \int_M \rho \om^{n-1}}{\int_M \om^n }= \frac{n}{2} \frac{ \bigl[ c_1 ( T^{1,0}M) \cup c_1 (L)^{n-1}\bigr]}{\bigl[ c_1(L)^n \bigr] }
\end{gather}

For any $t \in [0,1]$, let $j_t = \pi ( \ga_t)_* j_0$ and let us introduce three unitary maps
\begin{itemize}
\item $ U_t = (\ga_t)_* :\quant_k ( j_0) \rightarrow \quant_k ( j_t)$
\item  $ \mathcal P_t : \quant_k ( j_0) \rightarrow \quant_k ( j_t)$ is the
  parallel transport along the path $[0,t] \ni s \rightarrow j_s$
  \item  $\mathcal T_t : \quant_k ( j_t) \rightarrow \quant_k ( j_0)$  is the
    parallel transport along the geodesic joining $j_t$ to $j_0$.
  \end{itemize}
Then  $ \Psi_k ( \ga_t ) = \mathcal T_t \circ U_t = \eta_t \circ \xi_t $
with $$\eta_t = \mathcal T_t \circ \mathcal P_t, \qquad   \xi_t = \mathcal P_t^{-1} \circ U_t.$$
$ (\eta_t)$ and $(\xi_t)$ are smooth path of $\op{U}  ( \Hilb_k)$. We denote by $( \tilde{\eta}_t)$ and $( \tilde{\xi}_t)$ their lift to $\widetilde{\op{U}}  ( \Hilb_{k})$ starting for the identity. Of course, $\widetilde{\Psi}_{k} ( \ga) = \tilde{\eta}_1\tilde{\xi}_1$. 

Define the map $S :[0,1]^2 \rightarrow \J , (s,t) \rightarrow j_t^s $ such that for any $t$, $ s\rightarrow j^s_t$ is the geodesic joining $j_0$ to $j_t$. In the sequel we assume that $k$ is larger than $k (S)$.

\begin{prop} \label{sec:est1} We have  
$$\widetilde {\det} ( \tilde{\eta}_1 ) = -\frac{1}{2} \Bigl( \frac{k}{2 \pi } \Bigr)^{n}  \int_S \si + \bigo ( k^{n-1})$$ 
\end{prop}

\begin{proof} 
$\eta_1$ is the parallel transport along the boundary of $S$. So
$$ \widetilde{\det} ( \tilde{\eta}_1) = \int_{[0,1]^2} \op{tr} R_k ( a_t^s,
b_t^s ) \;  ds \wedge dt $$
where $a_t^s = \partial j_t^s/\partial s$ and $b_t^s = \partial j_t^s/\partial t$.
By the result \eqref{eq:Foth_uribe_curvature} of \cite{FoUr}, we have 
$$ R_k ( a_t^s, b_t^s ) = -\tfrac{1}{2} \Pi_k \si_{j_t^s} ( a_t^s, b_t^s ) + \bigo ( k^{-1})$$
Furthermore, the $\bigo$ is uniform with respect to $s$ and $t$. In particular $\Pi_k \si_{j_t^s} ( a_t^s, b_t^s )$ is uniformly bounded with respect to $s,t$ and $k$. To conclude we use that  $$\op{tr} \Pi_k f \Pi_k = \Bigl( \frac{k}{2 \pi} \Bigr)^{n} \int_M f \mu + \bigo ( k^{n-1}),$$ the $\bigo$ being uniform if $\sup |f|$ remains bounded.  
\end{proof}
Introduce now the Kostant-Souriau operators: for any $f \in \Ci (M)$,
 $$K_k ( f) = f + \frac{1}{ik} \nabla_X^{L^k} : \Ci ( M , L^k) \rightarrow \Ci ( M , L^k) $$
where $X$ is the Hamiltonian vector field of $f$. In the next proposition, we prove that $\xi_t^{-1}$ is the solution of a Schr\"odinger equation.
\begin{prop} \label{prop:Schrodinger_parallel}
Let $(H_t)$ be the Hamiltonian generating $(\ga_t)$ and $\phi_t = \pi ( \ga_t)$. We have
$$  \frac{i}{k} \frac{d}{dt} \bigl( \xi_t^{-1} \bigr) = - \Pi_k ( j_0) K_{k} ( H_t \circ \phi_t ) \; \xi_t^{-1}.$$
\end{prop}
This result has been proved in \cite[Proposition 4.3]{FoUr}, with a mistake however: the Hamiltonian $K_{k} ( H_t \circ \varphi_t)$ is replaced by $K_{k} ( H_t)$. 

In the case $H_t = H$ is time-independent, we have $H \circ \varphi_t = H$ and  $\xi_t$ commutes with $\Pi_k ( j_0) K_{k} ( H)$, we deduce that
$$ \frac{i}{k} \frac{d}{dt}  \xi_t  =  \Pi_k ( j_0) K_{k} ( H) \xi_t .$$
But this does not hold for a general time dependent $(H_t)$. However,  observe that     $- H_t \circ \phi_t$ is the generating Hamiltonian of $\ga_t^{-1}$.

\begin{proof} 
Let $V_t :\Ci ( M , L^k) \rightarrow \Ci ( M , L^k)$ be the push-forward by $\ga_t$. It is part of Kostant-Souriau theory that 
\begin{gather} \label{eq:KS_Schrodinger}
 \frac{i}{k} \dot{V}_t = K_{k} ( H_t) V_t
\end{gather}
$U_t$ being the restriction of $V_t$ to $\quant_k ( j_0)$, we have $\xi_t ^{-1} = U_t^{-1} \mathcal P_t = V_t^{-1}  \mathcal P_t$. Derivating, we get 
\begin{gather} \label{eq:unedeplus}
 \frac{d}{dt} \xi_t^{-1} = - V_t^{-1} \dot{V}_t V_t ^{-1} \mathcal P_t + V_t ^{-1} \dot{ \mathcal{P}}_t .
\end{gather}
Here, to give a meaning to $\dot{\mathcal P}_t$, we consider that $\mathcal P_t$ takes its value in $\Ci ( M , L^k)$. $t \rightarrow \mathcal P_t$ being the parallel transport along the path $t \rightarrow j_t$, we have $\Pi(j_t) \dot {\mathcal P}_t =0$. 
Since $\Pi ( j_t)  V_t = V_t \Pi ( j_0) $, we have $V_t^{-1} \Pi ( j_t) = \Pi ( j_0) V_t^{-1}$ and  consequently $\Pi ( j_0) V_t^{-1} \dot{\mathcal P}_t = V_t^{-1} \Pi ( j_t) \dot{\mathcal P}_t = 0 $. So by (\ref{eq:unedeplus}), 
\begin{gather*} 
  \frac{i}{k} \frac{d}{dt} \xi_t^{-1} = -  \frac{i}{k} \Pi ( j_0) V_t^{-1} \dot{V}_t V_t ^{-1} \mathcal P_t = -  \Pi ( j_0) V_t^{-1} K_k ( H_t) \mathcal P_t \\  = -  \Pi ( j_0) K_k ( H_t \circ \phi_t) V_t^{-1} \mathcal P_t = -  \Pi ( j_0) K_k ( H_t \circ \phi_t) \xi_t^{-1}
\end{gather*}
where we have used (\ref{eq:KS_Schrodinger}) and the fact that $ V_t^{-1} K_k ( H_t) =  K_k ( H_t \circ \phi_t) V_t^{-1}$.
\end{proof}

\begin{lemma} 
For any $f \in \Ci (M, \R)$, 
\begin{xalignat}{2} \label{eq:trace_moyenne}
\begin{split}
 \op{tr} & ( \Pi_k ( j_0) K_k (f) \Pi_k ( j_0)  ) \\
 &= \Bigl (\frac{k}{2 \pi} \Bigr) ^{n}  \Bigl( \Bigl(1+ \frac{\la'}{k} \Bigr) \int_M f \mu  
 + \frac{1}{2k} \int_M  \overline f  S(j_0) \mu  \Bigr)   + \bigo ( k^{n-2})
\end{split}
\end{xalignat}
where $\la'$ is the constant \eqref{eq:lambda} and $\overline{f} = f - \int_M f \mu / \int_M \mu$.
\end{lemma}
\begin{proof} 
Recall that  $\op{tr} ( T_k (f)  ) =  \int_M f B_k (x) \mu$ where $B_k $ is the restriction to the diagonal of the Schwartz kernel of $\Pi_k ( j_0)$. Since
$$ B_k = \Bigl (\frac{k}{2 \pi} \Bigr) ^{n}  \Bigl( 1 + \frac{ S( j_0)}{2k} + \bigo ( k ^{2}) \Bigr), $$
and by Tuynman formula \cite{Tuy} $\op{tr} ( T_k (f)  ) = \op{tr} ( \Pi_k ( j_0) K_k (f) \Pi_k ( j_0)   )$, we obtain 
\begin{gather*} 
\op{tr} ( \Pi_k ( j_0) K_k (f) \Pi_k ( j_0)   ) = \Bigl (\frac{k}{2 \pi} \Bigr) ^{n} \int_M f \Bigl( 1+    \frac{S( j_0)}{2k}  \Bigr) \; \mu  + \bigo ( k^{n-2})
\end{gather*}
which rewritten with the normalised Hamiltonian $\overline{f}$ gives (\ref{eq:trace_moyenne}).
\end{proof}

\begin{prop} \label{prop:detxi}
We have
$$ \widetilde{\det} ( \tilde{\xi}_1 ) = - \Bigl (\frac{k}{2 \pi} \Bigr) ^{n}   \Bigl( (k +  \la ') \ca( H_t)   + \frac{1}{2} \int_0^1 \int_M  \overline H_t  S(j_t) \mu \; dt \Bigr) + \bigo ( k^{n-1}) 
$$
\end{prop}
\begin{proof} 
By Proposition \ref{prop:Schrodinger_parallel}, we have $i k^{-1} \dot{\xi}_t = \xi_t \Pi_k ( j_0) K_{k} ( H_t \circ \phi_t ) $.  So 
\begin{gather*} 
 \frac{i}{ k} \frac{d}{dt} \ln (  \det  ( \xi_t ) ) = \frac{i}{ k}  \op{tr} ( \xi_t ^{-1} \dot{\xi}_t ) = \op{tr} \Pi_k ( j_0) K_{k} ( H_t \circ \phi_t )\Pi_k ( j_0)  \\
= \Bigl (\frac{k}{2 \pi} \Bigr) ^{n}  \Bigl( (1+ \frac{\la'}{k}) \int_M H_t \mu  + \frac{1}{2k} \int_M  \overline{H}_t  S(j_t) \mu  \Bigr) + \bigo ( k^{n-2})
\end{gather*}
by (\ref{eq:trace_moyenne}) where we have used that $\mu$ is preserved by $\phi_t$ and $S(j_0) \circ \phi_t^{-1} = S( j_t)$. To conclude, we use $\widetilde{\det} ( \tilde{\xi}_1 ) = \frac{1}{i} \int_0^1 \frac{d}{dt} \ln (  \det  ( \xi_t ) ) dt$.  
\end{proof}
Now \eqref{eq:toprove} is a consequence of the definition (\ref{eq:def_quasi_morphism}) of $\Sh_{j_0}$, Proposition \ref{prop:detxi}, Proposition \ref{sec:est1} and the fact that $\widetilde{\Psi}_{k} ( \ga) = \tilde{\eta}_1\tilde{\xi}_1$.

\section{Proof of Corollary \ref{cor:quasimorphism_prod} and Theorem \ref{theo:intro2}} \label{sec:proof-theor-2}

The first consequence of \eqref{eq:toprove} is corollary
\ref{cor:quasimorphism_prod}:  there exists $C>0$ such that for any $\al, \be \in \pathspace (\preq )$,
\begin{gather} \label{eq:cor}
\tilde{d} ( \widetilde{\Psi}_k ( \al) \widetilde{\Psi}_k ( \be) ,
\widetilde{\Psi}_k ( \al \be) ) \leqslant C + \bigo ( k)
\end{gather}
with a $\bigo$ depending on $\al, \be$. 

\begin{proof} 
By Proposition \ref{prop:estimation_universal_cover}, since $\widetilde \det$ is a morphism, we have
$$ \tilde{d} ( \widetilde{\Psi}_k ( \al) \widetilde{\Psi}_k ( \be) , \widetilde{\Psi}_k ( \al \be) )  \leqslant \frac{ \bigl|  \widetilde{\det} \bigl( \widetilde{\Psi}_{k} ( \al ) \bigr) +  \widetilde{\det} \bigl( \widetilde{\Psi}_{k} ( \be ) \bigr) -  \widetilde{\det} \bigl( \widetilde{\Psi}_{k} ( \al \be ) \bigr)  \bigr| }{ \dim \Hilb_{k}} + 2 \pi $$
The result follows from \eqref{eq:toprove} by using that $\ca$ is a morphism, $\Sh_{j_0}$ a quasi-morphism (\ref{eq:quasi-morphism_equation}) and $\dim \Hilb_{k} = \bigl( k/ 2 \pi \bigr)^{n} ( 1 + \bigo ( k^{-1}))$. 
\end{proof}

Introduce the map $\wt \Phi_{k}^{\op{KS}}  : \pathspace ( \preq ) \rightarrow
\widetilde{\op{U}}(\Hilb_k)$ defined  as $\wt \Phi_{k}$ in Section
\ref{sec:semicl-results} except that we use the Kostant-Souriau operators instead of the usual Toeplitz operator.  So $\wt \Phi_{k}^{\op{KS}}  (\ga_t) = \widetilde{W}_1$ where $(\widetilde{W}_t)$ is the lift of the solution $(W_t)$ of the Schr\"odinger equation
$$ W'_t = \frac{k}{i} \Pi_k ( j_0) K_k ( H_t) W_t, \qquad W_0 = \op{id}_{\Hilb_k}$$
$(H_t)$ being the generating Hamiltonian of $(\ga_t)$. 

\begin{theo}  \label{sec:theo_KS}
  For any path $\ga$ in $\pathspace( \preq)$, $\widetilde \Phi_k ^{\op{KS}}( \ga) = \widetilde{\Psi}_k ( \ga) + \bigo (1)$.
\end{theo}
\begin{proof} As in Section \ref{sec:proof-theor-1}, write $\widetilde{\Psi}_{k} ( \ga) =
  \tilde{\eta}_1\tilde{\xi}_1$.  On one hand, by Proposition
\ref{sec:est1}, Proposition \ref{prop:estimation_universal_cover} and the fact
that $\dim \Hilb_{k} = \bigl( k/ 2 \pi \bigr)^{n} ( 1 + \bigo ( k^{-1}))$, we
have $\tilde{d} (\tilde{\eta}_1 ,\op{id}) = \bigo (1)$. On the other hand, by Proposition \ref{prop:Schrodinger_parallel}, $
\tilde{\xi}_1 ^{-1}  = \wt \Phi_{k}^{\op{KS}} ( \ga^{-1}) $.
So by the right-invariance of $\tilde d$, 
\begin{gather} \label{eq:dsds}
\tilde d (\widetilde{\Psi}_k ( \ga)  ,\widetilde \Phi_k ^{\op{KS}}(
\ga^{-1})^{-1} ) = \bigo (1) .
\end{gather}
Introduce the notation $ \op{I}_k(
\ga) := \Phi_k ^{\op{KS}}(
\ga^{-1})^{-1}$. We deduce easily from \eqref{eq:cor} and \eqref{eq:dsds} that
for any path $\al, \be \in \pathspace ( \preq)$
 \begin{gather} \label{eq:ewew}
\tilde{d} ( \op{I}_k  ( \al)    \op{I}_k  ( \be) ,
\op{I}_k ( \al \be) ) = \bigo (1)
\end{gather}
Consequently
\begin{xalignat*}{2}
  \tilde{d} ( \op{I}_k ( \ga) , \Phi_k ^{\op{KS}} ( \ga) )  & = \tilde{d} (
  \op{I}_k ( \ga) ,  \op{I}_k ( \ga^{-1})^{-1} ) \\
   & = \tilde{d} ( \op{I}_k (\ga)
  \op{I}_k( \ga^{-1}), \op{id} ) \qquad \text{by right-invariance} \\
  & = \tilde{d} (\op{I}_k ( \op{id}), \op{id} ) + \bigo (1) \qquad \text{by }
  \eqref{eq:ewew} \\
   & = \bigo (1)
\end{xalignat*}
The result follows now from \eqref{eq:dsds}.
\end{proof}

We can now prove theorem \ref{theo:intro2}, that is
$$ 
\widetilde \Phi_k ( \ga) = \widetilde{\Psi}_k ( \ga) + \bigo (1)$$
\begin{proof}
  By Tuynman formula \cite{Tuy}, $\Pi_k(j_0) K_k ( f) = T_k (f)  + \bigo (
  k^{-1})$. So $\tilde{d} ( \Phi_{k}^{\op{KS}} ( \ga ) , \Phi_{k} ( \ga ) ) =
  \bigo (1)$. And we conclude with Theorem \ref{sec:theo_KS}.
\end{proof}

\subsection*{Acknowledgements}
Leonid Polterovich drew my attention to the paper \cite{Sh} and got me started with this work.   
I thank also Vincent Humili\`ere, Sohban Seyfadini, Egor Shelukhin and Alejandro Uribe for useful discussions. 

\bibliographystyle{alpha}
\bibliography{biblio}

\bigskip

\begin{tabular}{l}
Laurent Charles, \\  Sorbonne Universit\'e, CNRS, \\ 
 Institut de Math\'{e}matiques 
 de Jussieu-Paris Rive Gauche, \\
 IMJ-PRG, F-75005 Paris, France. 
\end{tabular}

\end{document}